\documentclass[A4paper,reqno,oneside,12pt]{amsart}
\usepackage{fullpage}
\usepackage{setspace}
\usepackage{amsmath, amssymb, amsfonts, amsthm, bbm, mathtools, mathrsfs}
\usepackage{tikz}
\usepackage{subcaption}
\usepackage{verbatim,enumitem,graphics}
\usepackage{xcolor}
\usepackage[backref=page]{hyperref}
\hypersetup{%
    colorlinks,
    linkcolor={red!50!black},
    citecolor={green!50!black},
    urlcolor={blue!50!black}
}
\usepackage{dsfont}
\usepackage{lineno}
\usepackage[normalem]{ulem}

\allowdisplaybreaks

\theoremstyle{plain}
\newtheorem{theorem}{Theorem}
\newtheorem*{theorem*}{Theorem}
\newtheorem{lemma}[theorem]{Lemma}

\newtheorem{corollary}[theorem]{Corollary}
\newtheorem*{corollary*}{Corollary}

\newtheorem{problem}[theorem]{Problem}

\theoremstyle{definition}
\newtheorem{claim}[theorem]{Claim}
\newtheorem{remark}[theorem]{Remark}
\newtheorem*{remark*}{Remark}

\newcommand{\N}{\mathbb{N}}

\def\d{\delta}

\def\l{\ell}
\def\e{\epsilon}

\newcommand{\set}[1]{\left\{ #1 \right\}}

\begin{document}

\setstretch{1.27}

\title{Recursive upper bounds for the vertex online Ramsey game with applications to hypergraph Ramsey numbers}

\author{D\'aniel Dob\'ak}
\address{Department of Mathematics, Emory University, Atlanta, GA, 30322, USA}
\email{daniel.dobak@emory.edu}

\author{Eion Mulrenin}
\address{Department of Mathematics, Emory University, Atlanta, GA, 30322, USA}
\email{eion.mulrenin@emory.edu}



\begin{abstract}
    The classical recursive upper bound on hypergraph Ramsey numbers due to Erd\H{o}s and Rado states that for $2 \leq k < s \leq t$,
    \[
        r_k(s,t) \leq 2^{\binom{r_{k-1}(s-1,t-1)}{k-1}}.
    \]
    In 2010, Conlon, Fox, and Sudakov introduced the so-called {\it vertex online Ramsey numbers} $\tilde{r}(s,t)$ for graphs to obtain a quantitative improvement over this bound when $k=3$.
    
    In this note, we show that the natural hypergraph generalization $\tilde{r}_k(s,t)$ of the vertex online Ramsey numbers satisfy an improved recurrence
    \[
        \tilde{r}_k(s,t) \leq 2^{(1+o(1))\tilde{r}_{k-1}(s-1,t-1)}.
    \]
    We obtain several corollaries from this, including a lower-order improvement to the best known quantitative upper bounds for hypergraph Ramsey numbers and an improvement to the above recursive bound of Erd\H{o}s and Rado.
\end{abstract}

\maketitle


\section{Introduction}
\label{section: intro}

For positive integers $k< t_1, \dots, t_q$, we denote by $r_k(t_1, \dots, t_q)$ the $k$-uniform $q$-color Ramsey number of $t_1, \dots, t_q$, that is, the least integer $n$ so that every $q$-coloring of $E(K_n^{(k)})$ contains a copy of $K_{t_c}^{(k)}$ the edges of which all receive color $c$, for some $c \in [q]$.
The study of these numbers goes back to the seminal paper of Ramsey~\cite{Ramsey}, but the problem of finding effective bounds for Ramsey numbers was first studied by Erd\H{o}s and Rado~\cite{ER52}, who 
implicitly proved a recursive bound of the form
\begin{equation}
\label{eq: stepping-down}
    r_k(t_1,\ldots,t_q) \leq 2^{\binom{r_{k-1}(t_1-1,\ldots,t_q-1)}{k-1}}.
\end{equation}
As a corollary, they gave a tower-type upper bound for such hypergraph Ramsey numbers,
greatly improving over the bound from Ramsey's proof, which was not even primitive recursive.
More precisely, we recursively define the {\it tower function} $T_i(x)$ by setting $T_0(x) = x$ and $T_{i+1}(x) = 2^{T_i(x)}$ (thus the subscript is the number of 2's in the tower); the recurrence in~\eqref{eq: stepping-down} for two colors, e.g., then gives that $r_k(s,t) \leq T_{k-1}((2+o(1))(s+t))$.

In this note, we investigate upper bounds for another hypergraph Ramsey problem, namely the {\it $k$-uniform vertex online Ramsey game}.
This was introduced by Conlon, Fox, and Sudakov in 2010~\cite{CFS10} for graphs (i.e., $k=2$) in order to give a quantitative improvement over the bound given by~\eqref{eq: stepping-down} for three-uniform hypergraphs.
A hypergraph analogue of the game was first described by Fox, Pach, Sudakov, and Suk~\cite{FPSS12}, who studied it for monotone paths,
and by Gasarch, Parrish, and Sinai~\cite{GPS}, who studied it for cliques.
Mubayi and Suk~\cite{MS20} also introduced another variant in their work on the Erd\H{o}s--Hajnal hypergraph Ramsey problem.

The $q$-color $k$-uniform game is played by two players, builder and painter, starting on the ordered vertex set $V = \{v_i: i \in \N\}$ of an empty $k$-uniform hypergraph, and proceeds in rounds as follows. 
\begin{itemize}
    \item At the beginning of round $i$, the $i$th vertex $v_i$ in the ordering of $V$ is revealed;
    \item During round $i$, builder can choose to build any number of edges of the form\\
    $\{v_{i_1}, ..., v_{i_{k-1}}, v_i\}$ for $i_1 < \dots < i_{k-1} < i$, in any order;
    \item Whenever builder decides to build an edge, painter must immediately assign it a color from $[q]$;
    \item In each round $i \geq k$, builder must build at least one new edge\footnote{Some versions of the game do not include this last rule, though to our knowledge, every known strategy which gives good bounds adheres to it.};
    \item Builder's goal is to force painter to create a monochromatic $K_{t_c}^{(k)}$ in some color $c \in [q]$, whereas painter's objective is to delay this as much as possible.
\end{itemize}
Since builder always has a winning strategy, namely, to build a clique on $r_k(t_1, \dots, t_q)$ vertices, the {\it $k$-uniform vertex online Ramsey number} $\tilde{r}_k(s,t)$ is then defined to be the minimum number of edges builder must draw to force painter to create a monochromatic $K_{t_c}^{(k)}$ in color $c$ for some $c \in [q]$.

It is worth mentioning that the more general {\it online Ramsey game}, introduced independently by Beck~\cite{B93} and Kurek and Ruci\'nski~\cite{KR05}, differs from the one studied here in that vertices are not revealed one at a time, and instead builder can draw whichever edges she wants at each step. 
For more information on this version, we refer the interested reader to the papers~\cite{C09, CFGH19} and the references therein.

Conlon, Fox, and Sudakov~\cite[Lemma 2.2]{CFS10} proved that for the graph vertex online Ramsey game, builder always has a strategy which uses at most $\binom{s+t-2}{s-1}$ vertices and at most\\
$(s+t-4)\binom{s+t-2}{s-1}$ edges to force painter to draw a red $K_s$ or a blue $K_t$.
In particular,
\begin{equation}
\label{eq: CFS-2-uniform-online}
    \tilde{r}_2(s,t) \leq (s+t-4) \cdot \binom{s+t-2}{s-1},
\end{equation}
and they subsequently use this to give an upper bound\footnote{For off-diagonal Ramsey numbers where $s$ is fixed and $t \to \infty$, they get a further improved bound with a slight alteration to their builder strategy.}
\begin{equation}
\label{eq: CFS-3-uniform}
    r_3(s,t) \leq 2^{\tilde{r}_2(s-1,t-1)} +1 \leq 2^{(s+t-4) \cdot \binom{s+t-2}{s-1}} + 1.
\end{equation}
After finishing the first draft of our paper, we became aware of an exposition by Gasarch, Parrish, and Sinai~\cite{GPS}. In this work, they describe a generalization of the Conlon-Fox-Sudakov method for hypergraphs of higher uniformity, but did not obtain an improvement for hypergraph Ramsey numbers. They subsequently raised the question of whether the vertex online Ramsey game for hypergraphs could be used to obtain a further improvement for $r_k(t,t)$.
Our main result and its first corollary below address this question, giving a recursive bound for $\tilde{r}_k(s,t)$ which yields much sharper recurrence than~\eqref{eq: stepping-down} for hypergraph vertex online Ramsey numbers.
\begin{theorem}
\label{thm: online bound}
    Let $2 \leq k < s, t$ and let $m = \tilde{r}_{k-1}(s-1,t-1)$.
    Then builder has a strategy in the $k$-uniform vertex online Ramsey game for cliques of size $s$ and $t$ using at most $2^m + k - 2$ vertices and at most $m \cdot 2^m$ edges.
    In particular,
    \begin{equation}
    \label{eq: ramsey-bound}
        r_k(s,t) \leq 2^{\tilde{r}_{k-1}(s-1,t-1)} + k - 2
    \end{equation}
    and
    \begin{equation}
    \begin{split}
    \label{eq: online bound}
        \tilde{r}_k(s,t) 
        &\leq \tilde{r}_{k-1}(s-1,t-1) \cdot 2^{\tilde{r}_{k-1}(s-1,t-1)} \\[1 ex]
        &= 2^{(1+o(1))\tilde{r}_{k-1}(s-1,t-1)},
    \end{split}
    \end{equation}
    where the $o(1)$ term goes to $0$ as $\max\set{s,t}\to\infty$.
\end{theorem}

It is worth pointing out that bounds of the form~\eqref{eq: ramsey-bound} were previously observed in several of the aforementioned papers~\cite{CFS10, FPSS12, GPS}.
We decided to include it in the statement of the theorem because it is an immediate consequence of our proof, which takes a slightly different route from this prior work (wherein it was proved by a modification of the Erd\H{o}s--Rado argument), and it will be used in the proofs of our corollaries.
On the other hand,~\eqref{eq: online bound} is, to our knowledge, the first nontrivial recursive-type bound for the hypergraph vertex online Ramsey game, and as such might be considered our main contribution.

This sharper recurrence for hypergraph vertex online Ramsey numbers quickly yields two corollaries.
First of all, we obtain a quantitative improvement to the best upper bound for diagonal hypergraph Ramsey numbers $r_k(t,t)$ as follows.

\begin{corollary}
\label{cor: quantitative-bound}
    Let $3 \leq k < t$. Then
    \begin{equation*}
        r_k(t,t)\leq T_{k-2}\left((1+o(1))2t\binom{2(t-k+1)}{t-k+1}\right).
    \end{equation*}
\end{corollary}

The proof follows quickly from Theorem \ref{thm: online bound}, by first applying~\eqref{eq: ramsey-bound} and then iterating~\eqref{eq: online bound} $k-3$ times, finally using the bound \eqref{eq: CFS-2-uniform-online} for graph vertex online Ramsey numbers. 
It also extends more generally to give an improvement for certain regimes of multicolor hypergraph Ramsey numbers, but we omit these details.


\begin{remark}
    This matches the current best upper bound for $r_3(t,t)$ due to Conlon, Fox and Sudakov; for $r_k(t,t)$ when $k \geq 4$, the previously known best upper bounds came from applying the recursion~\eqref{eq: stepping-down} starting with the bound for $r_3(t-k+2,t-k+2)$ in~\eqref{eq: CFS-3-uniform}, giving
    \begin{equation*}
        r_k(t,t) \leq T_{k-2}\left((1+o(1))6t\binom{2(t-k+1)}{t-k+1}\right).
    \end{equation*}
\end{remark}

\begin{remark}
    In light of exciting recent progress on two-uniform diagonal Ramsey numbers~\cite{BBC+, CGMS}, we would also like to point out that 
    small exponential improvements to $r_2(n)$, together with \eqref{eq: stepping-down}, would not yield better bounds for $r_k(t,t)$. 
    In particular, $2^{\binom{r_2(t-1,t-1)}{2}}$ would only be a better bound for $r_3(t,t)$ than that given by Corollary~\ref{cor: quantitative-bound} if $r_2(t,t) \leq (2+o(1))^t$.
\end{remark}

Our second corollary is an exponential improvement on the recursive upper bound of Erd\H{o}s and Rado~\eqref{eq: stepping-down} for roughly half of all valid combinations of $k < s \leq t$.
To our knowledge, this is the first nontrivial improvement to the recursive bound in~\eqref{eq: stepping-down}.
In particular, we prove the following.

\begin{corollary}
\label{cor: recursive-improvement}
    For all $\d \in (0,1]$ and $2 \leq k < s \le t$, at least one of the following two inequalities holds:
    \begin{equation*}
        r_k(s,t) \leq 2^{\frac{\binom{r_{k-1}(t-1,s-1)}{k-1}}{\delta k}}
        \hspace{1cm}
        \text{or}
        \hspace{1cm}
        r_{k+1}(s+1, t+1) \leq 2^{r_k(s, t)^{(1+o(1))\delta k}},
    \end{equation*}
    where the $o(1)$ term goes to $0$ as $t\to \infty$.
\end{corollary}

This result improves the Erd\H{o}s-Rado upper bound $r_k(s,t) \leq 2^{\binom{r_{k-1}(s-1,t-1)}{k-1}}$ in~\eqref{eq: stepping-down}.
Note that for all combinations $2 \leq k < s \leq t$, we either get a small improvement over the bound for $r_k(s,t)$, or a substantial improvement over the bound for $r_{k+1}(s+1, t+1)$.

Corollary \ref{cor: recursive-improvement} also holds more generally for $q$ colors: for this more general statement and proof, see Section \ref{section: corollary proofs}. \\

\noindent
\subsection*{Organization} The rest of the paper is organized as follows:
in Section \ref{section: proof of online bound} we prove Theorem \ref{thm: online bound}, in Section \ref{section: corollary proofs} we prove Corollaries \ref{cor: quantitative-bound} and \ref{cor: recursive-improvement}, and in Section \ref{section: conclusion} we discuss some concluding remarks. 


\section{Proof of Theorem~\ref{thm: online bound}}
\label{section: proof of online bound}

In this section, we state and prove Theorem \ref{thm: online bound} in full generality, for $q$ colors.

\begin{theorem}
    Let $2 \leq k < t_1, \dots, t_q$, and let $m = \tilde{r}_{k-1}(t_1-1, \dots, t_q-1)$.
    Then builder has a strategy in the $q$-color $k$-uniform online Ramsey game for cliques of size $t_1, \dots, t_q$ using at most $q^m+ k - 2$ vertices and at most $m \cdot q^m$ edges.
    In particular,
    \begin{equation} \label{eq: q-color ramsey bound}
        r_k(t_1, \dots, t_q) \leq q^m + k - 2,
    \end{equation}
    and
    \begin{equation} \label{eq: q-color online bound}
        \tilde{r}_k(t_1, \dots, t_q) \leq q^{m + \log_q m}.
    \end{equation}
\end{theorem}

\begin{proof}
Let $2 \leq k < t_1, \dots, t_q$ and $q \geq 2$ be given, and recall that we assume there is a builder strategy in the $(k-1)$-uniform $q$-color vertex online Ramsey game which forces a
monochromatic $K_{t_c-1}^{(k-1)}$ in some color $c \in [q]$ using
at most $m = \tilde{r}_{k-1}(t_1-1,\dots, t_q-1)$ edges against {\it any} painter strategy.
Our goal is to show that builder has a strategy in the $k$-uniform $q$-color vertex online Ramsey game which forces painter to draw a
monochromatic $K_{t_{c'}}^{(k)}$ in some color $c' \in [q]$
and which uses at most $q^m + k-2$ vertices and at most $m \cdot q^m$ edges.
Roughly speaking, the idea of the $k$-uniform builder strategy is that, once the $j$th vertex $v_j$ is revealed to begin the $j$th round, builder will run the $(k-1)$-uniform strategy as a ``sub-routine" on a certain set $R(v_j)$ of previously revealed vertices.
On this set, builder will run this optimal $(k-1)$-uniform strategy against an auxiliary $(k-1)$-uniform online coloring $\Delta_j$ which is determined by painter's $k$-uniform coloring $\Delta$.
By design, once there is a monochromatic $K_{t_c-1}^{(k-1)}$ under one of these auxiliary colorings $\Delta_j$, there will also be a monochromatic $K_{t_c}^{(k)}$ under $\Delta$.

More formally, we define builder's $k$-uniform strategy by recursively constructing a map $\Phi$ from the ordered vertex set $V = \{v_1, v_2, \dots\}$ into 
$[q]^*$, the set of all finite strings over the alphabet $[q]$.
For each $j \geq 1$, the $j$th round of the game begins with the vertex $v_j$ being revealed, and after playing this round, builder will define $\Phi(v_j) \in [q]^*$ and an associated set of previously revealed vertices $R(v_j) \subseteq \{v_1, \dots, v_{j-1}\}$.
$\Phi$ will be an injection on $V \setminus \{v_1, \dots, v_{k-2}\}$, and $R(v_j)$ will be the vertex set on which the $(k-1)$-uniform coloring $\Delta_j$ is defined.

To initialize, let $v_1, ..., v_{k-1}$ be the first $k-1$ vertices revealed in the game, and set
$\Phi(v_j) := \e$, the empty string, for each $j \in [k-1]$.
Now let $j \geq k$ and let $v_j$ be the vertex revealed at the beginning of the $j$th round, and recursively assume that builder has already defined $\Phi(v_i)$ and $R(v_i)$ for all $i < j$.
We define $\Phi(v_j)$ and $R(v_j)$ by the following recursive sub-procedure:
\begin{itemize}
    \item {\bf Initial step.} Set $\Phi_0(v_j) := \e$ and $R_0(v_j) := \emptyset$.
    \item {\bf Recursion step.} Suppose for $f \geq 0$ that builder has already defined $\Phi_f(v_j) \in [q]^*$ and a set of previously revealed vertices $R_f(v_j) \subseteq \{v_1, \dots, v_{j-1}\}$ of size $f$.
    Furthermore, suppose that builder has run $f$ rounds of the $(k-1)$-uniform strategy on this set $R_f(v_j)$ against an auxiliary $(k-1)$-uniform coloring $\Delta_j$ that she has defined, in the order that vertices were added to $R_f(v_j)$.
    
    If there exists a vertex $v_i \notin R_f(v_j)$ with $i < j$ such that $\Phi_f(v_j) = \Phi(v_i)$,
    builder picks $i$ which is minimal with this property\footnote{Strictly speaking, $\Phi$ will be an injection on $V \setminus \{v_1, \dots, v_{k-2}\}$, so this is only necessary for $f \in [k-2]$.}
    and initiates round $f+1$ of the $(k-1)$-uniform vertex online Ramsey game on $R_f(v_j) \cup \{v_i\}$ with $v_i$ as the revealed vertex.
    At each step of the $(f+1)$st round, by the rules of the game, the $(k-1)$-uniform strategy will build an edge $e \cup \{v_i\}$, where $e \subseteq R_f(v_j)^{(k-2)}$, and ask for it to be colored immediately by the auxiliary coloring $\Delta_j$.
    Builder decides the color $\Delta_j(e \cup \{v_i\})$ by first, in the $k$-uniform game, revealing the edge $e \cup \{v_i, v_j\}$ (recall that we are in the $j$th round of the $k$-uniform game) and then having painter assign it a color $\Delta(e \cup \{v_i,v_j\}) \in [q]$ immediately.
    Once this has occurred, builder sets
    \begin{equation*}
        \Delta_j(e \cup \{v_i\}) := \Delta(e \cup \{v_i, v_j\}).
    \end{equation*}
    Now that the edge $e \cup \{v_i\}$ has received a color under $\Delta_j$ in the $(k-1)$-uniform game, we proceed to the next step of round $f+1$ of this game.

    At the conclusion of the $(f+1)$st round, builder sets:
    \begin{itemize}
        \item $R_{f+1}(v_j) := R_f(v_j) \cup \{v_i\}$; and
        \item $\Phi_{f+1}(v_j)$ to be the string in $[q]^*$ obtained by appending, in order, the colors $\Delta_j(e \cup \{v_i\})$ used at each step in the $(f+1)$st round of the $(k-1)$-uniform game to the word $\Phi_f(v_j)$.
    \end{itemize}

    Eventually, for some $f$, there will be no $i < j$ with $\Phi(v_i) = \Phi_f(v_j)$.
    When this happens, builder sets $\Phi(v_j) := \Phi_f(v_j)$ and $R(v_j) := R_f(v_j)$, and ends round $j$ of the $k$-uniform game.
\end{itemize}

\begin{lemma}
\label{lem: main}
    Let $k \leq i < j$, and suppose $v_i \in R(v_j)$.
    Then the following claims hold:
    \begin{itemize}
        \item[(i)] $R(v_i) = \{v_\l \in R(v_j): \l < i\}$; and
        \item[(ii)] If $e \subseteq R(v_i) \cap R(v_j) \overset{(i)}{=} R(v_i)$ has size $k-1$, then builder drew the edge $e \cup \{v_i\}$ if and only if she drew the edge $e \cup \{v_j\}$, and in this event,
        \begin{equation*}
            \Delta(e \cup \{v_i\}) = \Delta(e \cup \{v_j\}). 
        \end{equation*}
    \end{itemize}
\end{lemma}

\begin{proof}
    {\it (i)}
    By construction, $v_i \in R(v_j)$ if and only if there exists $f \geq 0$ with $\Phi_f(v_j) = \Phi(v_i)$, so it's enough to prove that for this same choice of $f$, $R_f(v_j) = R(v_i)$.
    Write $\Phi_f(v_j) = \Phi(v_i)$ as both $w_1 \cdots w_f$ and $w_1' \cdots w_g'$, where $w_h$ and $w_h'$ are, respectively, the substrings appended after round $h$ of the $(k-1)$-uniform game when builder was defining $\Phi_f(v_j)$ and $\Phi(v_i)$.
    Note that since $\Phi$ is an injection on $V \setminus \{v_1, \dots, v_{k-2}\}$, when $k \leq i < j$ we can write
    \begin{equation}
    \label{eq: Rs}
        R_f(v_j) = \Phi^{-1}(w_1) \cup \dots \cup \Phi^{-1}(w_1 \cdots w_f)
        \hspace{0.25cm}
        \text{and}
        \hspace{0.25cm}
        R(v_i) = \Phi^{-1}(w_1') \cup \dots \cup \Phi^{-1}(w_1' \cdots w_g').
    \end{equation}
    We now prove by induction that $f = g$ and $w_h = w_h'$ for all $h \in [f]$, which suffices to prove $R_f(v_j) = R(v_i)$ by~\eqref{eq: Rs}.

    For the base case, note that $w_h = w_h' = \e$ for all $h \in [k-2]$, as builder's $(k-1)$-uniform strategy cannot build any edges until there are at least $k-1$ vertices revealed.
    Now suppose that $w_1 = w_1'$, \dots, $w_h = w_h'$ for some $h$ with $k-2 \leq h < \min\{f,\l\}$, and that moreover, in the first $h$ rounds, builder's $(k-1)$-uniform strategy picked the same edges in the same order against both $\Delta_j$ and $\Delta_i$, and that each such edge was assigned the same color by $\Delta_j$ and $\Delta_i$.
    At this stage, there must have been some vertex $v_\l$ with $k-1 \leq \l < i, j$ such that $\Phi(v_\l) = w_1 \cdots w_h = w_1' \cdots w_h'$, and this $v_\l$ is unique, as again, $\Phi$ is an injection on vertices past the first $k-2$.
    In both $(k-1)$-uniform games, this means $v_\l$ was the vertex revealed at the beginning of the $(h+1)$st round, and $w_{h+1}$ and $w_{h+1}'$ are the sequences of colors that were used by $\Delta_j$ and $\Delta_i$, respectively, on edges of the form $e \cup \{v_\l\}$.

    Since the first $h$ rounds of the $(k-1)$-uniform game against $\Delta_j$ and $\Delta_i$ were identical, when $v_\l$ is revealed to begin the $(h+1)$st round, the deterministic $(k-1)$-uniform strategy will build the same edges, and since $\Phi_f(v_j) = \Phi(v_i)$, it must be that each time a new edge is built, it receives the same color from $\Delta_j$ and $\Delta_i$.
    And again, by the fact that the $(k-1)$-uniform strategy is deterministic, the $(h+1)$st round will conclude in both games at exactly the same step, whence $w_{h+1}=w_{h+1}'$ and the $(h+1)$st round is identical in both games.

    {\it (ii)} now follows immediately from the fact that the first $f$ rounds of both $(k-1)$-uniform games are identical.
\end{proof}

\begin{claim}
\label{clm: endgame}
    If there is a vertex $v_n$ with $\Phi(v_n)$ of length $m$, then painter has drawn a monochromatic $K_{t_c}^{(k)}$ in some color $c \in [q]$ under $\Delta$, and thus the game ends.
\end{claim}

\begin{proof}
    First, note that a label of length $m$ implies that builder's $(k-1)$-uniform strategy built $m$ $(k-1)$-uniform edges in $R(v_n)$, which were colored under $\Delta_n$, and by assumption, this strategy guaranteed a monochromatic $K_{t_c-1}^{(k-1)}$ in some color $c \in [q]$ once $m$ edges were drawn.
    Let $t=t_c$ and let $U = \{u_1, \dots, u_{t-1}\} \subseteq R(v_n)$ be the vertices of this $K_{t-1}^{(k-1)}$ which is monochromatic under $\Delta_n$.
    Our claim is that the set $U \cup \{v_n\}$ induces a monochromatic $K_t^{(k)}$ in color $c$ under $\Delta$.
    First, note that 
    \begin{equation*}
        \Delta(e \cup \{v_n\}) = c
    \end{equation*}
    for all $e \in U^{(k-1)}$;
    indeed, this follows since each such edge $e$ was colored with $c$ by $\Delta_n$.

    Now let $i_1 < \dots < i_k \in [t-1]$.
    Since $u_{i_k} \in R(v_n)$, Lemma~\ref{lem: main}(i) implies that
    $R(u_{i_k}) \supseteq \{u_{i_1}, \dots, u_{i_{k-1}}\}$, and so by Lemma~\ref{lem: main}(ii), since builder drew the edge $\{u_{i_1}, \dots, u_{i_{k-1}}, v_n\}$, she also drew the edge $\{u_{i_1}, \dots, u_{i_k}\}$, and we have
    \begin{equation*}
        \Delta(u_{i_1}, \dots, u_{i_{k-1}}, u_{i_k}) = \Delta(u_{i_1}, \dots, u_{i_{k-1}}, v_n) = c,
    \end{equation*}
    as promised.
\end{proof}

\begin{claim}
\label{clm: injection}
    Let $V = \{v_1, \dots, v_n\}$ be the set of all vertices revealed in the game.
    Then $\Phi$ is an injection from $V \setminus \{v_1, \dots, v_{k-2}, v_n\}$ into $[q]^{\leq m-1}$.
\end{claim}

\begin{proof}
    That $\Phi$ is an injection on this set of vertices is immediate from its construction.
    The claim that the codomain is $[q]^{\leq m-1}$ follows from Claim~\ref{clm: endgame}, since the game ends as soon as vertex $v_n$ receives a label $\Phi(v_n)$ with at least $m$ letters.
\end{proof}

To finish, we just have to upper-bound the total number of vertices and edges we have used.
By Claim~\ref{clm: injection}, we immediately have
\begin{equation*}
    |V| \leq k-1+\sum_{i=0}^{m-1} q^i \leq q^m + k - 2.
\end{equation*}

For the edge count, observe that edges with largest vertex $v_j$ are in one-to-one correspondence with letters in the string $\Phi(v_j)$.
Indeed, if an edge $e = \{v_{i_1}, \dots, v_{i_k}\}$ was built by builder with $i_1 < \dots < i_k$ and painter assigned $\Delta(e) = c$, then the letter $c$ appears in the word $\Phi(v_{i_k})$ at the point when builder's $(k-1)$-uniform strategy picked the edge $e \setminus \{v_{i_k}\}$ during round $i_k$ of the $k$-uniform game.
Thus, the number of edges built by builder in the $k$-uniform game is exactly
\begin{eqnarray*}
    \sum_{j=1}^n |\Phi(v_j)|
    &\leq& m + \sum_{i=1}^{m-1} i \cdot q^i\\
    &\leq& m \cdot q^m,
\end{eqnarray*}
as required.
\end{proof}


\section{Proofs of the Corollaries}
\label{section: corollary proofs}
For the reader's convenience we restate Corollary \ref{cor: quantitative-bound}.

\begin{corollary*}
    Let $3 \leq k < t$. Then
    \begin{equation*}
        r_k(t,t)\leq T_{k-2}\left((1+o(1))2t\binom{2(t-k+1)}{t-k+1}\right).
    \end{equation*}
\end{corollary*}

\begin{proof}
     We apply~\eqref{eq: ramsey-bound} and then iterate~\eqref{eq: online bound} $k-3$ times to get
    \begin{eqnarray*}
        r_k(t,t) 
        &\stackrel{\eqref{eq: ramsey-bound}}{\leq}& 2^{(1+o(1))\tilde{r}_{k-1}(t-1,t-1)}\\
        &\stackrel{\eqref{eq: online bound}}{\leq}& 2^{T_{k-3}((1+o(1))\tilde{r}_2(t-k+2,t-k+2))}\\
        &\stackrel{\eqref{eq: CFS-2-uniform-online}}{\leq}& T_{k-2}\left((1+o(1))2t\binom{2(t-k+1)}{t-k+1}\right).
    \end{eqnarray*}
\end{proof}

Finally, we state and prove Corollary \ref{cor: recursive-improvement} in full generality, for $q$ colors.

\begin{corollary*}
\label{cor: recursive-improvement q-color}
    For all $\d \in (0,1]$ and $2 \leq k < t_1, \dots, t_q$, at least one of the following two inequalities holds:
    \begin{equation*}
        r_k(t_1, \dots, t_q) \leq q^{\frac{\binom{r_{k-1}(t_1-1,\ldots,t_q-1)}{k-1}}{\delta k}}
        \hspace{1cm}
        \text{or}
        \hspace{1cm}
        r_{k+1}(t_1+1, \dots, t_q+1) \leq q^{r_k(t_1, \dots, t_q)^{(1+o(1))\delta k}},
    \end{equation*}
    where the $o(1)$ term goes to $0$ as $\max_i\{t_i\}\to \infty$.
\end{corollary*}

\begin{proof}
    The two cases depend on whether $\tilde{r}_k(t_1, \dots, t_q) \leq r_k(t_1, \dots, t_q)^{\mu k}$, where
    \[\mu=\delta\left(1+\frac{\log_q \tilde{r}_{k-1}(t_1-1, \dots, t_q-1)}{\tilde{r}_{k-1}(t_1-1, \dots, t_q-1)}\right)=(1+o(1))\delta.\]
    If this inequality holds, then
    \begin{eqnarray*}
        r_{k+1}(t_1+1, \dots, t_q+1) 
        &\stackrel{\eqref{eq: q-color ramsey bound}}{\leq}& q^{\tilde{r}_k(t_1, \dots, t_q)} + k-1\\
        &\leq& q^{r_k(t_1, \dots, t_q)^{(1+o(1))\delta k} }.
    \end{eqnarray*}
    Otherwise, we have $\tilde{r}_k(t_1, \dots, t_q) > r_k(t_1, \dots, t_q)^{\mu k}$, and so
    \begin{eqnarray*}
        r_k(t_1, \dots, t_q) 
        &\leq& \tilde{r}_k(t_1, \dots, t_q)^{1/\mu k}\\
        &\stackrel{\eqref{eq: q-color online bound}}{\leq}& q^{\left(\tilde{r}_{k-1}(t_1-1, \dots, t_q-1)+\log_q \tilde{r}_{k-1}(t_1-1, \dots, t_q-1)\right)/\mu k}\\
        &=& q^{\frac{\tilde{r}_{k-1}(t_1-1, \dots, t_q-1)}{\delta k}}\\
        &\leq& q^{\frac{\binom{r_{k-1}(t_1-1,\ldots,t_q-1)}{k-1}}{\delta k}},
    \end{eqnarray*}
    where in the last line we used the trivial bound $\tilde{r}_{k-1}(t_1-1, \dots, t_q-1) \leq \binom{r_{k-1}(t_1-1,\ldots,t_q-1)}{k-1}$.
    This completes the proof.
\end{proof}


\section{Concluding remarks}
\label{section: conclusion}

In this note, we showed that
\begin{equation*}
    r_k(s,t) \leq 2^{\tilde{r}_{k-1}(s-1,t-1)} + k - 2
    \qquad
    \text{and}
    \qquad
    \tilde{r}_k(s,t) \leq 2^{(1+o(1))\tilde{r}_{k-1}(s-1,t-1)}.
\end{equation*}
It would of course be desirable to improve on the trivial bound $\tilde{r}_{k-1}(s,t) \leq \binom{r_{k-1}(s,t))}{k-1}$ in such a way that 
Theorem~\ref{thm: online bound} gave a genuine improvement over~\eqref{eq: stepping-down} for all combinations $k < s \leq t$.
We believe that even something a bit stronger should be true (cf. the main result in~\cite{C09}).

\begin{problem}
    Show that for $2 \leq k < t$, we have that $\tilde{r}_k(t,t) = o\left(\binom{r_k(t,t)}{k}\right)$ as $t \to \infty$.
\end{problem}

Another natural question is whether our recursive bound for vertex online Ramsey numbers is sharp.
For example, is it true that $\tilde{r}_k(t,t) = 2^{\Theta(\tilde{r}_{k-1}(t-1,t-1))}$?
The stepping-up lemma does not adapt in an obvious way to provide a matching lower bound, and so a proof of such a statement appears to require some new ideas.

\subsection*{Acknowledgments}
We would like to thank Dhruv Mubayi and Vojt\v{e}ch R\"odl for insightful conversations.


\end{document}